%% file: main.tex
\documentclass[10pt,reqno]{amsart}
\usepackage[T1]{fontenc}
\usepackage[utf8]{inputenc}
\usepackage{lmodern}
\usepackage{amsmath,amssymb,mathtools,xcolor}
\usepackage[expansion=false]{microtype}
\usepackage[margin=1.5in]{geometry}
\usepackage{hyperref}

\numberwithin{equation}{section}
\newtheorem{theorem}{Theorem}[section]
\newtheorem{lemma}[theorem]{Lemma}
\newtheorem{proposition}[theorem]{Proposition}
\theoremstyle{remark}
\newtheorem{remark}[theorem]{Remark}

\DeclareMathOperator{\diam}{diam}
\DeclareMathOperator{\Int}{int}
\newcommand{\R}{\mathbb R}
\newcommand{\Acoef}{\mathcal A}
\newcommand{\Bcoef}{\mathcal B}
\newcommand{\Gflux}{\mathcal G}
\newcommand{\eps}{\varepsilon}
\newcommand{\abslt}[1]{\left|#1\right|}
\newcommand{\norm}[1]{\lVert#1\rVert}

\providecommand{\eat}[1]{}

\allowdisplaybreaks[1]

\title[Enclosing rectangles and Neumann eigenvalues]{Minimum-Area Enclosing Rectangles and the First Neumann Eigenvalue of Planar Convex Domains}

\author[Xiaolong Li]{Xiaolong Li}
\address{Department of Mathematics and Statistics, Auburn University, Auburn, AL, 36849, USA}
\email{xil0005@auburn.edu}
\thanks{The first author was supported in part by NSF-DMS \#2553660 and a start-up grant at Auburn University.}

\author{Jin Sun}
\address{School of Mathematical Sciences, Fudan University, Shanghai, 200433, China}
\email{jsun22@m.fudan.edu.cn}

\author{Kui Wang}
\address{School of Mathematical Sciences, Soochow University, Suzhou, 215006, China}
\email{kuiwang@suda.edu.cn}

\date{}

\begin{document}

\begin{abstract}
Let $\Omega$ be a bounded convex domain in the plane and let $R$ be a
minimum-area rectangle containing $\overline\Omega$. We prove that if the
ratio of the shorter side of $R$ to its longer side is at most $\rho \approx 0.443485$,
then the first nonzero Neumann eigenvalue of $\Omega$ is bounded from below by that of $R$. Moreover, equality holds
if and only if $\Omega$ coincides with the interior of $R$. The proof is based on a
flux identity for the means of a test function over vertical sections of
$\Omega$, combined with weighted estimates involving the distributional second
derivatives of the boundary functions. Finally, we show that the
threshold $\rho\approx0.443485$ is nearly optimal by constructing, for every
ratio in $(0.457994,1/2)$, trapezoids that are arbitrarily close to their
minimum-area enclosing rectangle and for which the inequality fails.
\end{abstract}
\maketitle

\section{Introduction}

Let $\Omega\subset\R^2$ be a bounded open convex domain. We consider the
Neumann eigenvalue problem
\begin{equation}\label{eq:neumann}
\begin{cases}
-\Delta u=\mu u & \text{in }\Omega,\\[2pt]
\partial_{\nu}u=0 & \text{on }\partial\Omega,
\end{cases}
\end{equation}
understood in the weak sense, where $\nu$ denotes the outward unit normal
to $\partial\Omega$. The first nonzero eigenvalue of \eqref{eq:neumann}
is characterized by
\begin{equation}\label{eq:variational}
\mu_1(\Omega)=
\inf_{u\in H^1(\Omega)\setminus\{0\}}\left\{
\frac{\int_\Omega|\nabla u|^2}{\int_\Omega u^2}: \int_\Omega u(x)\, dx=0\right\},
\end{equation}
so that $1/\mu_1(\Omega)$ is the optimal constant in the Poincar\'e
inequality for functions of mean zero on $\Omega$. Throughout the paper,
all functions are real-valued.

Estimating $\mu_1(\Omega)$ in terms of geometric quantities of $\Omega$,
such as its area and diameter, is a classical problem. The
Szeg\H{o}--Weinberger inequality \cite{Szego54,Weinberger56} asserts that
among planar domains of prescribed area, the disk maximizes $\mu_1$. For
convex domains, Payne and Weinberger \cite{PW60} proved the sharp lower
bound
\[
\mu_1(\Omega)\ge\frac{\pi^2}{\diam(\Omega)^2},
\]
where the constant is approached by rectangles whose shorter sides tend
to zero. Quantitative refinements in terms of the width have recently
been obtained by Amato, Bucur and Fragal\`a \cite{ABF24} and by Wang and
Xu \cite{WX24}; see also \cite{ABF26} for nonlinear extensions. For various extensions to curved spaces, see \cite{LY80}, \cite{ZY84}, \cite{Kroger99}, \cite{BQ00}, \cite{AC13}, \cite{AN12}, \cite{Ni2013}, \cite{LW21JGA, LW21MRL, LW21TAMS},\cite{ZW2017}, just to mention a few. 

For a rectangle $R$ with side lengths $L\ge W>0$, it is well known that
$\mu_1(\Int R)=\pi^2/L^2$. This suggests looking for geometric conditions
under which the diameter in the Payne--Weinberger inequality can be
replaced by the longer side of a rectangle containing $\Omega$. Inclusion
alone is not enough, since Neumann eigenvalues are not monotone with
respect to inclusion, even among convex domains. In this direction,
Freitas and Kennedy \cite{FK25} proved that for bounded convex planar
domains $\Omega_1\subset\Omega_2$,
\[
\mu_1(\Omega_1)\ge
\frac{\pi^2}{4j_{0,1}^2}\mu_1(\Omega_2),
\]
where $j_{0,1}$ is the first positive zero of the Bessel function $J_0$
and the factor is optimal for general convex inclusions. A related
obstruction comes from the work of Cavallina, Funano, Henrot, Lemenant,
Lucardesi and Sakaguchi \cite{CFHLLS25}, who showed that for
a fixed bounded open convex planar domain $\Omega_0$, the infimum of $\mu_1$ over
convex subdomains of $\Omega_0$ equals $\pi^2/\diam(\Omega_0)^2$. When $\Omega_0$ is a
rectangle with sides $L$ and $W$, thin rectangles approaching a diagonal
of $\Omega_0$ therefore have eigenvalues tending to $\pi^2/(L^2+W^2)$, which is
smaller than $\pi^2/L^2$.

In view of these examples, we impose a condition on the enclosing
rectangle itself. By a minimum-area enclosing rectangle of $\Omega$ we
mean a rectangle, of arbitrary orientation, which contains
$\overline\Omega$ and has the least area among all such rectangles. Existence follows from the continuity of the two orthogonal widths as
functions of the orientation and compactness of the set of orientations.
Such a rectangle need not be unique, and we simply fix one of them. Since the function $f(s)=s\cot s$ is strictly decreasing from $1$ to $0$ on $(0,\pi/2)$, there exists a unique $\rho\in(0,1/2)$ satisfying
\begin{equation}\label{eq:rho}
\pi\rho\cot(\pi\rho)=\frac{1}{4}.
\end{equation}
The main result of this paper states the following.

\begin{theorem}\label{thm:rho}
Let $\Omega$ be a bounded open convex set in the plane, and let $R$ be
a minimum-area enclosing rectangle of $\Omega$, with side
lengths $L\ge W>0$. If $W/L\le\rho$, then
\[
\mu_1(\Omega)\ge\frac{\pi^2}{L^2}=\mu_1(\Int R).
\]
Equality holds if and only if $\Omega=\Int R$.
\end{theorem}

\begin{remark}
Numerically, $\rho\approx0.443485$, so Theorem~\ref{thm:rho} applies
whenever the minimum-area enclosing rectangle is sufficiently elongated.
The numerical value is not used anywhere in the proof. We only use the
defining equation \eqref{eq:rho} and the elementary bounds in
Lemma~\ref{lem:rho-constants}.
\end{remark}

After a rigid motion and scaling, we may assume that
\[
\overline\Omega\subset R_r:=[0,1]\times[0,r],
\qquad r=\frac WL\le\rho,
\]
where $R_r$ is a minimum-area enclosing rectangle of $\Omega$. Thus, proving Theorem
\ref{thm:rho} amounts to proving
\begin{equation}\label{eq:comparison}
\mu_1(\Omega)\ge\pi^2
\end{equation}
and characterizing the case of equality.

The threshold $\rho$ in Theorem~\ref{thm:rho} is close to optimal. In
Appendix~\ref{sec:near-optimality}, we show that there is a constant
$r_*\approx0.457993$ with the following property. For
$r\in(r_*,1/2)$, consider the trapezoids
\[
\Omega_{r,s}
=\{(x,y):0<x<1,\ 0<y<r(1-sx)\},
\qquad 0<s<1.
\]
The rectangle $R_r$ is a minimum-area enclosing rectangle of each
$\Omega_{r,s}$, and $\overline{\Omega_{r,s}}$ converges to $R_r$ in the
Hausdorff distance as $s\to0$. Nevertheless,
\[
\mu_1(\Omega_{r,s})<\pi^2
\]
for all sufficiently small $s>0$. Hence, the optimal threshold in
Theorem~\ref{thm:rho} lies in $[\rho,r_*]$, an interval of length less
than $0.015$.  It remains an interesting problem to determine the optimal threshold.

This paper is organized as follows. In Section~\ref{sec:prelim}, we
prove the contact separation estimate, collect some one-dimensional
inequalities, and establish the bounds for $\rho$.
Section~\ref{sec:flux} contains the flux identity, the
transverse estimate and the shear estimate. We prove
Theorem~\ref{thm:rho}, including the characterization of the equality
case, in Section~\ref{sec:proof}. The appendix gives a direct variational proof of the trapezoid
counterexamples.

\bigskip

\textbf{AI disclosure.}  The core mathematical results and proofs in this paper are the original contributions of the authors. All essential strategies regarding Theorem \ref{thm:rho} were formulated by the authors. OpenAI’s ChatGPT was used solely as an assistive tool to identify the trapezoids example in Appendix A; all arguments and calculations have been rigorously verified by the authors.

\section{Preliminaries}\label{sec:prelim}

\subsection{Contact separation}

Let $0<r<1$, and let $\Omega\subset(0,1)\times(0,r)$ be an open convex
set whose projection onto the $x$-axis is $(0,1)$. Since $\Omega$ is convex, we may write
\begin{equation*}
\Omega=\{(x,y):0<x<1,\ a(x)<y<b(x)\},
\end{equation*}
and
\begin{equation}\label{eq:profiles}
 h(x)=b(x)-a(x),\quad\text{and}\quad c(x)=\frac{a(x)+b(x)}{2},
\end{equation}
where $a(x)$ is convex, $b(x)$ is concave, and $h(x)$ is positive and concave on
$(0,1)$. These functions extend continuously to $[0,1]$ and satisfy
$0\le a\le b\le r$. Note that the heights at the endpoints may
vanish and that the boundary slopes may be unbounded near $x=0$ and
$x=1$. We denote by
\[
\nu_a=a'',\qquad \nu_b=-b'',\qquad \nu_a+\nu_b=-h''
\]
the nonnegative Radon measures on $(0,1)$.

If $R_r$ is a minimum-area enclosing rectangle of $\Omega$, then each side of
$R_r$ intersects $\overline\Omega$. Consequently, the bottom and top contact sets
\[
I_-(\Omega)=\{x\in[0,1]:(x,0)\in \overline\Omega\},
\qquad
I_+(\Omega)=\{x\in[0,1]:(x,r)\in \overline\Omega\}
\]
are nonempty closed intervals, possibly reduced to points.

In the proof of Lemma~\ref{lem:contacts} we need an elementary fact for
differentiating extrema. Suppose that $f_\theta$ and $g$ are continuous
functions on $\overline\Omega$ such that $f_\theta=f_0+\theta g+o(\theta)$ uniformly on
$\overline\Omega$ as $\theta\rightarrow0$. Then
\[
\left.\frac{d}{d\theta}\right|_{0+}\max_{\overline\Omega} f_\theta
=\max_{\{p\in \overline\Omega:f_0(p)=\max_{\overline\Omega}f_0\}}g(p).
\]

\begin{lemma}\label{lem:contacts}
Let $R_r$ be a minimum-area enclosing rectangle of $\Omega$.
If $I_-(\Omega)$ and $I_+(\Omega)$ are disjoint, replace $x$ by
$1-x$ if necessary so that
\[
\alpha:=\max I_-(\Omega)<\beta:=\min I_+(\Omega).
\]
Then
\[
\beta-\alpha\le r^2.
\]
Moreover, $a'b'\le0$ almost everywhere outside
$[\alpha,\beta]$. If $I_-(\Omega)\cap I_+(\Omega)\neq \emptyset$, then $a'b'\le0$
almost everywhere on $(0,1)$.
\end{lemma}

\begin{proof}
Suppose $I_-(\Omega)\cap I_+(\Omega) = \emptyset$ and set $d=\beta-\alpha$. For $\theta\in\R$, introduce the
rotated coordinates
\[
X_\theta(x,y)=x\cos\theta+y\sin\theta,
\qquad Y_\theta(x,y)=-x\sin\theta+y\cos\theta,
\]
and define the widths of $\Omega$ in these two orthogonal directions by
\[
\begin{aligned}
W_X(\theta)&:=\max_{(x,y)\in \overline\Omega}X_\theta(x,y)
             -\min_{(x,y)\in \overline\Omega}X_\theta(x,y),\\
W_Y(\theta)&:=\max_{(x,y)\in \overline\Omega}Y_\theta(x,y)
             -\min_{(x,y)\in \overline\Omega}Y_\theta(x,y).
\end{aligned}
\]
Thus, $W_X(0)=1$ and $W_Y(0)=r$. Their product
$\mathcal W(\theta):=W_X(\theta)W_Y(\theta)$ is the area of the smallest
enclosing rectangle whose sides are parallel to the rotated coordinate
axes. The minimum-area condition gives that for every $\theta\in\R$,
\[
\mathcal W(\theta)\ge\mathcal W(0)=r.
\]
We now compute the right derivatives at $\theta=0$. Note that uniformly on $\overline\Omega$,
\[
X_\theta=x+\theta y+O(\theta^2),
\qquad Y_\theta=y-\theta x+O(\theta^2).
\]
The maximizers of $Y_0$ are the top contacts and its
minimizers are the bottom contacts. Hence,
\[
\begin{aligned}
\left.\frac{d}{d\theta}\right|_{0+}\max_{\overline\Omega}Y_\theta
&=-\min I_+(\Omega)=-\beta,\\
\left.\frac{d}{d\theta}\right|_{0+}\min_{\overline\Omega}Y_\theta
&=-\max I_-(\Omega)=-\alpha.
\end{aligned}
\]
Therefore, $W_Y'(0+)=\alpha-\beta=-d$. Similarly, we have
\[
W_X'(0+)=\max\{y:(1,y)\in \overline\Omega\}-\min\{y:(0,y)\in \overline\Omega\}\le r.
\]
Differentiating the area product yields
\begin{align*}
0\le\mathcal W'(0+)
&=W_X'(0+)W_Y(0)+W_X(0)W_Y'(0+)\\
&=rW_X'(0+)-d\le r^2-d.
\end{align*}
Thus, we have $\beta-\alpha=d\le r^2$.

It remains to prove the sign conditions. Since $a$ attains its minimum at $\alpha$ and $b$ attains its maximum at $\beta$,
convexity forces $a'\le0, b'\ge0$ on $[0,\alpha]$, 
and $a'\ge0, b'\le0$ on $[\beta,1]$. Thus $a'b'\le0$ almost everywhere outside $[\alpha,\beta]$. 
If $I_-(\Omega)\cap I_+(\Omega) \neq \emptyset$, we choose
$x_0\in I_-(\Omega)\cap I_+(\Omega)$. Then the same sign conditions hold to the
left and to the right of $x_0$ respectively, and give $a'b'\le0$
almost everywhere on $(0,1)$. This argument also covers the cases
$x_0=0$ and $x_0=1$.
\end{proof}

\subsection{One-dimensional inequalities}

Let  $\phi(x)=\sin(\pi x)$. For any  $z\in H_0^1(0,1)$, a direct calculation gives
\[
|z'|^2-\pi^2z^2
=\phi^2\left|\left(\frac z\phi\right)'\right|^2
+\left(\frac{\phi'}\phi z^2\right)',
\]
and integrating above identity over $(0,1)$ yields 
\begin{equation}\label{eq:ground-state}
\int_0^1|z'|^2\, dx-\pi^2\int_0^1z^2\, dx
=\int_0^1\phi^2\left|\left(\frac z\phi\right)'\right|^2\, dx,
\end{equation}
We will  use  the above identity frequently in the following, and 
 also use the following Poincar\'e inequality for functions of mean zero. If $I$ is an interval of length $t$ and $v\in H^1(I)$ has mean zero on
$I$, then
\begin{equation}\label{eq:slice-Poincare}
\int_Iv^2\le\frac{t^2}{\pi^2}\int_I|v'|^2.
\end{equation}

For $z\in C_c^\infty(0,1)$, integration by parts and the Cauchy--Schwarz
inequality give
\[
\int_0^1\frac{z^2}{x^2}\,dx
=2\int_0^1\frac{zz'}x\,dx
\le2\left(\int_0^1\frac{z^2}{x^2}\,dx\right)^{1/2}
\left(\int_0^1|z'|^2\,dx\right)^{1/2}.
\]
By density, the resulting bound extends to $H_0^1(0,1)$. Adding its
reflection about $x=1/2$ gives the following Hardy's inequality
\begin{equation}\label{eq:Hardy}
\int_0^1\frac{z^2}{\min\{x,1-x\}^2}\,dx
\le8\int_0^1|z'|^2\,dx,
\qquad z\in H_0^1(0,1).
\end{equation}

\subsection{Bounds for the constant \texorpdfstring{$\rho$}{rho}}

\begin{lemma}\label{lem:rho-constants}
The constant $\rho$ defined by \eqref{eq:rho} satisfies
\begin{equation*}
	\rho^2<\frac{1}{5},
\end{equation*}
and 
\begin{align*}
     \kappa_\rho:=\frac{5\pi\rho}{7}\,\tan\frac{\pi\rho}{2}<1.
\end{align*}
\end{lemma}

\begin{proof}
Set $\gamma=\pi\rho$. Then $0<\gamma<\pi/2$ and
\[
\gamma\cot\gamma=\frac14.
\]
We first prove $\rho^2<1/5$. Since $f(s)=s\cot s$ is strictly decreasing on $(0,\pi/2)$ and
$f(\pi/3)=\pi/(3\sqrt{3})>1/4$, we obtain $\gamma>\pi/3$. 
For $0<x<1$, we have the elementary inequality 
\[
\tan x\le \frac{x}{1-x^2}.
\]
Let $x = \pi/2-\gamma<1$. Then
\[
\frac{1}{4\gamma}\le \frac{\pi/2-\gamma}{1-(\pi/2-\gamma)^2}.
\]
Therefore, we have
\[
3\gamma^2-\pi\gamma+1-\frac{\pi^2}{4}\le 0.
\]
Since $\rho=\gamma/\pi$ and $1/\pi^2>1/10$, we obtain
\[
3\rho^2-\rho-\frac{3}{20}<0,
\]
Thus,
\begin{equation}\label{eq:rho_upperbound}
\rho^2<\left(\frac{5+\sqrt{70}}{30}\right)^2
<\frac{1}{5}.
\end{equation}

By the half-angle formula and $\tan\gamma=4\gamma$, we have
\[
\tan\frac{\gamma}{2}
=\frac{\sqrt{1+16\gamma^2}-1}{4\gamma}.
\]
Thus,
\[
\frac{5\pi\rho}{7}\,\tan\frac{\pi\rho}{2}
=\frac{5\gamma}{7}\,\tan\frac{\gamma}{2}
=\frac{5}{28}\bigl(\sqrt{1+16\gamma^2}-1\bigr).
\]
Consequently,
\[
\kappa_\rho=\frac{5\pi\rho}{7}\,\tan\frac{\pi\rho}{2}
<\frac{5}{28}\left(\sqrt{33}-1\right)<1,
\]
where we used $\gamma^2=\pi^2\rho^2<2$.
\end{proof}

\section{Flux inequalities}\label{sec:flux}

Throughout this section we use the notation of Section~\ref{sec:prelim},
with $0<r\le\rho$. No smoothness of the boundary is assumed. The distributional
second derivatives of the boundary functions enter only through the measures
$\nu_a$ and $\nu_b$.

\subsection{The flux identity}

Fix $u\in H^1(\Omega)$ with $\int_\Omega u\, dx=0$. We decompose $u$ into its
sectional mean $m$ and a fluctuation $w$ which has mean zero on each
vertical section, that is,
\[
m(x)=\frac1{h(x)}\int_{a(x)}^{b(x)}u(x,y)\,dy,
\qquad w(x,y)=u(x,y)-m(x).
\]
where $h(x)$ is defined in \eqref{eq:profiles}.
For almost every $x$, $w\in H^1(a(x),b(x))$ and
satisfies $w_y=u_y$ on that section. 
Let $M$ and $Q$ denote the contributions of these two components to the
squared $L^2(\Omega)$ norm, namely
\[
M=\int_0^1hm^2,\qquad Q=\int_\Omega w^2.
\]
Since $w$ is orthogonal to constants on each section, we have
\[
\int_\Omega u^2=M+Q.
\]
We define the sectional flux $F$, its normalization $z$, and the squared
$L^2(0,1)$ norm $Z$ of $z$ by
\begin{equation}\label{eq:flux}
F(x)=-\int_0^xh(t)m(t)\,dt,
\qquad z(x)=\frac{F(x)}{\sqrt{h(x)}},
\qquad Z=\int_0^1z^2.
\end{equation}
Since $u$ has mean zero, $F(0)=F(1)=0$. The normalization by $\sqrt h$
is chosen so that the weight $h$ disappears from the leading derivative
term in the identity below. 
To measure how far $z$ is from being a multiple of $\phi(x)=\sin(\pi x)$, we put
\begin{equation}\label{eq:flux-positive-terms1}
E=\int_0^1\phi^2\left|\left(\frac z\phi\right)'\right|^2\, dx.
\end{equation}
We also put
\begin{equation}\label{eq:flux-positive-terms2}
V_a=\frac12\int_{(0,1)}\frac{z^2}{h}\,d\nu_a,
\qquad
V_b=\frac12\int_{(0,1)}\frac{z^2}{h}\,d\nu_b,
\quad V=V_a+V_b.
\end{equation}
where $V_a$ and $V_b$ are the weighted curvature
contributions of the lower and upper boundaries, respectively. The
curvature integrals are taken in the sense of measures on $(0,1)$, so
that integration by parts on compact subintervals remains valid when the
boundary functions are not smooth.

\begin{lemma}\label{lem:flux-positive}
The normalized flux $z$ defined in \eqref{eq:flux} belongs to $H_0^1(0,1)$. 
All the quantities in \eqref{eq:flux-positive-terms1} and \eqref{eq:flux-positive-terms2} are finite, and
\begin{equation}\label{eq:positive-identity}
M=\pi^2Z+E+V+\frac34\int_0^1\left(\frac{h'}h\right)^2z^2.
\end{equation}
\end{lemma}

\begin{proof}
On compact subintervals of $(0,1)$, the identity $F'=-hm$ can be
rewritten as
\[
\sqrt h\,m=-z'-\frac{h'}{2h}z.
\]
We choose truncation points $\eps$ and $1-\eps$ at which $h'$ has no
jump. Using the above identity and integrating by parts, we obtain
\begin{align}\label{eq:truncated-flux}
\int_\eps^{1-\eps}hm^2
=\int_\eps^{1-\eps}|z'|^2
+\frac34\int_\eps^{1-\eps}\left(\frac{h'}h\right)^2z^2+\frac12\int_{(\eps,1-\eps)}\frac{z^2}{h}\,d(-h'')
+\left[\frac{h'}{2h}z^2\right]_\eps^{1-\eps}.
\end{align}
We now show that the boundary term tends to zero. For
$0\le t\le x\le1/2$, the concavity of $h$ gives
$h(t)\le(1-t)h(x)/(1-x)\le2h(x)$. Hence, by Cauchy--Schwarz
inequality,
\begin{equation*}
z(x)^2\le\frac{1}{h(x)}\int_0^xh\int_0^xhm^2
\le2x\int_0^xhm^2=o(x).
\end{equation*}
Similarly, $z(x)^2=o(1-x)$ as $x\to1$. On the other hand, the concavity
of $h$ also gives
\[
-\frac1{1-x}\le\frac{h'}h\le\frac1x,
\]
so the boundary term in \eqref{eq:truncated-flux} tends to zero. All the
other terms on the right-hand side of \eqref{eq:truncated-flux} are nonnegative. Letting the truncated
intervals exhaust $(0,1)$, we see that these terms are finite and that
$z\in H_0^1(0,1)$. The identity \eqref{eq:positive-identity} then
follows from \eqref{eq:ground-state}.
\end{proof}

\subsection{The transverse estimate}

We now estimate the transverse part of the energy. For $0<t<1$, let
\begin{equation}\label{eq:AB}
\Acoef(t)=\frac{2\tan(\pi t/2)}{\pi t},
\qquad
\Bcoef(t)=\frac1{\pi^2t^2}-\frac{\cot(\pi t/2)}{2\pi t}.
\end{equation}
Both functions extend continuously to $t=0$, with values $1$ and $1/12$,
respectively. For the functions in \eqref{eq:profiles}, we set
\begin{equation*}
T=\int_0^1z^2\bigl(\Acoef(h)c'^2+\Bcoef(h)h'^2\bigr).
\end{equation*}

\begin{lemma}\label{lem:transverse}
The quantity $T$ is finite, and for every $u\in H^1(\Omega)$ with
$\int_\Omega u=0$,
\begin{equation}\label{eq:flux-energy}
M^2\le(Z+T)\left(\int_\Omega|\nabla u|^2-\pi^2Q\right).
\end{equation}
In particular, if $M\ge\pi^2(Z+T)$ for every mean-zero
$u\in H^1(\Omega)$, then $\mu_1(\Omega)\ge\pi^2$.
\end{lemma}

\begin{proof}
For $0<t<1$, let
\[
I_t=(-t/2,t/2),
\qquad
\mathcal H_t=\left\{v\in H^1(I_t):\int_{I_t}v=0\right\},
\]
and define the symmetric bilinear form
\[
q_t(v_1,v_2)
=\int_{I_t}(v_1'v_2'-\pi^2v_1v_2),
\qquad q_t(v)=q_t(v,v).
\]
By \eqref{eq:slice-Poincare},
\[
q_t(v)\ge(1-t^2)\int_{I_t}|v'|^2,
\qquad v\in\mathcal H_t.
\]
Thus, $q_t$ is positive definite on $\mathcal H_t$.

Consider the odd and even functions
\begin{align*}
    \psi_o(s)=\frac{\sin(\pi s)}{\pi\cos(\pi t/2)},
\end{align*}
and
\begin{align*}
  \psi_e(s)=-\frac{\cos(\pi s)}{2\pi\sin(\pi t/2)}
+\frac1{\pi^2t}.  
\end{align*}
Both functions belong to $\mathcal H_t$, and direct differentiation gives
\begin{align*}
  -\psi_o''-\pi^2\psi_o=0,  
\end{align*}
and
\begin{align*}
-\psi_e''-\pi^2\psi_e=-\frac 1 t,
\end{align*}
together with the endpoint conditions
\[
\psi_o'(\pm t/2)=1,
\qquad
\psi_e'(\pm t/2)=\pm\frac12.
\]
Integrating by parts, we obtain that for every $v\in\mathcal H_t$,
\[
q_t(\psi_o,v)=\int_{I_t}v',
\qquad
q_t(\psi_e,v)=\int_{I_t}\frac{s}{t}v'.
\]
By parity, $q_t(\psi_o,\psi_e)=0$, while integration by parts yields
\begin{align*}
  q_t(\psi_o)
=\psi_o(t/2)-\psi_o(-t/2)
=\frac2\pi\tan\frac{\pi t}{2}
=t\Acoef(t),  
\end{align*}
and
\begin{align*}
 q_t(\psi_e)
=\psi_e(t/2)
=\frac1{\pi^2t}-\frac{\cot(\pi t/2)}{2\pi}
=t\Bcoef(t).   
\end{align*}
Applying the Cauchy--Schwarz inequality for $q_t$ to
$q_t(\xi_1\psi_o+\xi_2\psi_e,v)$, we obtain
\begin{equation}\label{eq:dual-interval}
\left|\int_{I_t}\left(\xi_1+\frac{\xi_2}{t}s\right)v'(s)\,ds\right|^2
\le t\{\Acoef(t)\xi_1^2+\Bcoef(t)\xi_2^2\}\,q_t(v),
\qquad \xi_1,\xi_2\in\mathbb R.
\end{equation}

Since $a$ is convex, $b$ is concave and
both take values in $[0,r]$, we have
\begin{equation}\label{eq:a'_and_b'}
	|a'(x)|+|b'(x)|\le\frac{2r}{\min\{x,1-x\}}
\end{equation}
for almost every $x\in(0,1)$.
The coefficients $\Acoef$ and $\Bcoef$ extend continuously to $t=0$ and
are therefore bounded on $[0,r]$. Hence, Hardy's inequality
\eqref{eq:Hardy} and Lemma~\ref{lem:flux-positive} give
\begin{equation}\label{eq:T-Hardy}
T\le C_r\int_0^1\frac{z^2}{\min\{x,1-x\}^2}\,dx
\le C_r\int_0^1|z'|^2\,dx
\le C_rM,
\end{equation}
where $C_r>0$ denotes a constant depending only on $r$, whose value may
change from one occurrence to the next. In particular, $T$ is finite.

Assume first that $u$ is smooth in a neighborhood of $\overline\Omega$.
Parametrizing each vertical section by $y=c(x)+h(x)s$, we have
\[
m(x)=\int_{-1/2}^{1/2}u(x,c(x)+h(x)s)\,ds,
\]
and differentiation gives that for almost every $x\in(0,1)$,
\[
m'(x)=\frac1h\int_a^b
\left[u_x+\left(c'+\frac{h'}h(y-c)\right)u_y\right]\,dy.
\]
Since $F'=-hm$, integration by parts yields
\begin{equation}\label{eq:hm_2}
\int_\varepsilon^{1-\varepsilon}hm^2
=\int_\varepsilon^{1-\varepsilon}Fm'
-[Fm]_\varepsilon^{1-\varepsilon}.
\end{equation}
Since $m$ is bounded and $F(0)=F(1)=0$, the boundary term vanishes as
$\varepsilon\to0$. Moreover, the bounds used in \eqref{eq:T-Hardy} also
imply
\[
\begin{aligned}
\int_\Omega\left(\frac Fh\right)^2
\left[1+\left(c'+\frac{h'}h(y-c)\right)^2\right]
=\int_0^1z^2\left(1+c'^2+\frac{h'^2}{12}\right)\,dx
<\infty.
\end{aligned}
\]
Therefore, we can choose $\varepsilon\to0$ in \eqref{eq:hm_2}. Since
$w_y=u_y$, we obtain
\begin{equation}\label{eq:flux-pairing}
M=\int_0^1Fm'
=\int_\Omega\frac Fh
\left[u_x+\left(c'+\frac{h'}h(y-c)\right)w_y\right].
\end{equation}
Since $\int_\Omega(F/h)^2=Z$, the first term is
bounded in absolute value by $\sqrt Z\,\norm{u_x}_2$. For the second
term, we apply \eqref{eq:dual-interval} for almost every $x$ with
\[
t=h(x),\qquad v(s)=w(x,c(x)+s),
\qquad \xi_1=c'(x),\qquad \xi_2=h'(x).
\]
By Cauchy--Schwarz inequality, we have
\[
\abslt{\int_\Omega\frac Fh
	\left[\left(c'+\frac{h'}h(y-c)\right)w_y\right]}
\le\sqrt T\left(\int_\Omega(w_y^2-\pi^2w^2)\right)^{1/2}.
\]
Consequently, \eqref{eq:flux-pairing} gives
\[
M\le
\sqrt Z\left(\int_\Omega u_x^2\right)^{1/2}
+\sqrt T\left(\int_\Omega(w_y^2-\pi^2w^2)\right)^{1/2}.
\]
Applying the Cauchy--Schwarz inequality once more to these two terms and
using $w_y=u_y$, we obtain \eqref{eq:flux-energy} for every mean-zero
function $u$ which is smooth in a neighborhood of $\overline\Omega$.

For a general $u\in H^1(\Omega)$ with mean zero, we choose smooth mean-zero functions $u_n$
converging to $u$ in $H^1(\Omega)$, and we denote the corresponding
quantities by $m_n$, $z_n$, $M_n$, $Q_n$, $Z_n$ and $T_n$. Then
\[
\int_0^1h|m_n-m|^2
\le\int_\Omega|u_n-u|^2.
\]
Applying Lemma~\ref{lem:flux-positive} and \eqref{eq:T-Hardy} to
$u_n-u$, we get
\[
\begin{aligned}
\int_0^1|z_n'-z'|^2\,dx+\int_0^1|z_n-z|^2
\{\Acoef(h)c'^2+\Bcoef(h)h'^2\}\,dx
\le C_r\int_\Omega|u_n-u|^2.
\end{aligned}
\]
It follows that $M_n\to M$, $Q_n\to Q$, $Z_n\to Z$ and $T_n\to T$.
Therefore, \eqref{eq:flux-energy} passes to the limit.

Finally, suppose that $M\ge\pi^2(Z+T)$ for every mean-zero test function,
and fix a nonzero function $u\in H^1(\Omega)$ with mean zero. If $M>0$,
then \eqref{eq:flux-energy} implies
\[
M^2\le\frac M{\pi^2}
\left(\int_\Omega|\nabla u|^2-\pi^2Q\right),
\]
and hence
\[
\int_\Omega|\nabla u|^2\ge\pi^2(M+Q)
=\pi^2\int_\Omega u^2.
\]
If $M=0$, then $u$ has mean zero on almost every vertical section.
Since $h\le r<1$, applying \eqref{eq:slice-Poincare} on each section
gives
\begin{equation}\label{eq:zero-sectional-mean}
\int_\Omega|\nabla u|^2
\ge\frac{\pi^2}{r^2}\int_\Omega u^2
>\pi^2\int_\Omega u^2.
\end{equation}
In both cases, the variational characterization \eqref{eq:variational}
proves the last assertion.
\end{proof}

\subsection{The shear estimate}\label{subsec:curvature}

For the rest of this section, $R_r$ is a minimum-area enclosing
rectangle of $\Omega$. The quantity to be controlled is
\[
S=\int_0^1(a'b')_+z^2\,dx.
\]
It is finite by the slope bound in the proof of
Lemma~\ref{lem:transverse} and Hardy's inequality. Recall that
$E$, $V_a$ and $V_b$ are finite by Lemma~\ref{lem:flux-positive}.

The following argument uses the product of the two gaps between the
boundary functions and the horizontal sides of the rectangle. This
product vanishes at the contacts, and its second derivative relates
the shear directly to the boundary measures.

\begin{lemma}\label{lem:universal-shear}
Let $R_r$ be a minimum-area enclosing rectangle of $\Omega$. Then
\begin{equation}\label{eq:universal-shear}
S\le\frac5{14}r^2(E+V).
\end{equation}
\end{lemma}

\begin{proof}
If the contact intervals intersect, then $S=0$ by
Lemma~\ref{lem:contacts} and \eqref{eq:universal-shear} holds. Otherwise, after replacing $x$ by $1-x$ if necessary, we may assume as in Lemma~\ref{lem:contacts} that
\[
\alpha=\max I_-(\Omega),\qquad
\beta=\min I_+(\Omega),\qquad
0\le\alpha<\beta\le1,\qquad d=\beta-\alpha.
\]
By Lemma \ref{lem:contacts} and \ref{lem:rho-constants}, we have $d<1/5$.
On the interval $J=(\alpha,\beta)$, put
\[
p=a'b',\qquad g=a(r-b),\qquad
\ell=(r-b)a'+ab',
\]
where $g$ is the product of the distances from the lower and upper
boundaries to the bottom and top sides of $R_r$, and it vanishes at both
contacts.
Since $a$ is convex with minimum $a(\alpha)=0$ and $b$ is
concave with maximum $b(\beta)=r$, we have $a'\ge 0$ and $b'\ge 0$ almost
everywhere on $J$. This gives
\[
S=\int_J pz^2\,dx.
\]

Convexity and the supporting-line inequalities yield the following inequalities almost everywhere
on $J$.
\begin{equation}\label{eq:shear-tangents}
\begin{aligned}
&0\le a\le(x-\alpha)a',
&& a+(1-x)a'\le r,\\
&0\le r-b\le(\beta-x)b',
&& r-b+xb'\le r.
\end{aligned}
\end{equation}
Consequently,
\[
0\le g\le\frac{d^2}{4}p,\qquad 0\le\ell\le dp.
\]
A direct computation gives
\[
\begin{aligned}
(a+(1-x)a')(r-b+xb')-\ell
=(xa'-a)\bigl((1-x)b'-(r-b)\bigr)\ge0.
\end{aligned}
\]
Together with \eqref{eq:shear-tangents}, this yields
\begin{equation}\label{eq:shear-geometric-bounds}
\ell\le r^2,\qquad \ell^2\le r^2dp.
\end{equation}
In particular, $|g'|\le\ell\le r^2$. Since $g$ vanishes at $\alpha$ and $\beta$, its extension by zero outside $J$ is Lipschitz.

On $J$, the product rule for distributions gives
\[
g''=(r-b)\nu_a+a\nu_b-2p\,dx.
\]
Note that
\[
a'(\alpha+)\le\nu_a\{\alpha\}\quad(\alpha>0),
\qquad
b'(\beta-)\le\nu_b\{\beta\}\quad(\beta<1).
\]
Therefore, multiplying by $z^2$ and integrating by parts, we obtain
\[
S\le
\frac12\int_{[\alpha,\beta]\cap(0,1)}(r-b)z^2\,d\nu_a
+\frac12\int_{[\alpha,\beta]\cap(0,1)}az^2\,d\nu_b
+\int_J g'zz'\,dx.
\]
These computations remain valid when $\alpha$ or $\beta$ is an endpoint of $[0,1]$. Indeed,
$g'$ is bounded and $z$ has zero trace, so the boundary term
$g'z^2$ of the truncated integrals tends to zero.  Since
\[
(r-b)h\le(r-b)b\le\frac{r^2}{4},\qquad
ah\le a(r-a)\le\frac{r^2}{4},
\]
we obtain
\begin{equation}\label{eq:shear-ibp}
S\le\frac{r^2}{4}V+\int_Jg'zz'\,dx.
\end{equation}

It remains to estimate the last integral in \eqref{eq:shear-ibp}. 
Since $\phi=\sin(\pi x)$ is positive and concave on $(0,1)$, we have
\[
-\frac1{1-x}\le\frac{\phi'}{\phi}\le\frac1x.
\]
Combining these inequalities with \eqref{eq:shear-tangents}, we obtain
\[
\frac{\phi'}{\phi}a\le a',
\qquad
-\frac{\phi'}{\phi}(r-b)\le b'.
\]
Hence,
\begin{equation}\label{eq:shear-drift}
\left|g'-2\frac{\phi'}{\phi}g\right| = \left|(r-b)\Bigl(a'-\frac{\phi'}{\phi}a\Bigr)
-a\Bigl(b'+\frac{\phi'}{\phi}(r-b)\Bigr)\right|
\le\ell\le r\sqrt{dp}.
\end{equation}
Write $v=z/\phi$. Since $g$ vanishes at $\alpha, \beta$, integration by parts gives
\begin{equation}\label{eq:Jg'zz'}
\int_Jg'zz'\,dx
=-\int_Jg(\phi\phi')'v^2\,dx
+\int_J\left(g'-2\frac{\phi'}{\phi}g\right)z\phi v'\,dx.
\end{equation}
If $\alpha=0$ or $\beta=1$, the additional boundary term
$g(\phi'/\phi)z^2$ also vanishes there. Indeed, the Lipschitz bound on $g$
cancels the endpoint singularity of $\phi'/\phi$, and $z$ has zero trace.
The remaining integrals near the endpoints are finite by Hardy's
inequality.

Since $0\le g\le\frac{d^2}{4}p$ and $(\phi\phi')'=\phi'^2-\pi^2\phi^2\ge-\pi^2\phi^2$,
we have
\begin{equation*}
	-\int_Jg(\phi\phi')'v^2\,dx
	\le\pi^2\int_Jgz^2\,dx\le\frac{\pi^2d^2}{4}S.
\end{equation*}
By \eqref{eq:shear-drift} and Cauchy--Schwarz inequality,
we obtain
\begin{equation*}
	\int_J\left(g'-2\frac{\phi'}{\phi}g\right)z\phi v'\,dx
	\le r\sqrt d\left(\int_Jpz^2\,dx\right)^{1/2}
	\left(\int_0^1\phi^2|v'|^2\,dx\right)^{1/2}
	=r\sqrt{dSE}.
\end{equation*}
Combining above inequalities with Young's inequality, \eqref{eq:Jg'zz'} gives
\begin{equation*}
	\int_Jg'zz'\,dx
\le\frac{\pi^2d^2}{4}S+ dS+\frac{r^2}{4}E.
\end{equation*}
Substitution into \eqref{eq:shear-ibp} yields
\begin{equation}\label{eq:shear-general-remainder}
\left(1-d-\frac{\pi^2d^2}{4}\right)S
\le\frac{r^2}{4}(E+V).
\end{equation}
Since $d<1/5$, we have
\[
1-d-\frac{\pi^2d^2}{4}>\frac7{10},
\]
and \eqref{eq:universal-shear} follows.
\end{proof}

\section{\texorpdfstring{Proof of Theorem~\ref{thm:rho}}{Proof of the main theorem}}\label{sec:proof}

In this section, we combine the results of Sections~\ref{sec:prelim}
and~\ref{sec:flux} to prove Theorem~\ref{thm:rho}. We establish the remainder
inequality \eqref{eq:rho-remainder} directly on $\Omega$ and deduce the
lower bound, and then we use the vanishing of the nonnegative
remainder terms to characterize equality.

\begin{proof}[Proof of Theorem~\ref{thm:rho}]
We normalize the minimum-area enclosing rectangle to $R_r=[0,1]\times[0,r]$, where
$0<r\le\rho$. We fix a mean-zero function $u\in H^1(\Omega)$ and use the
notation of Section~\ref{sec:flux}. The identities
\[
c'^2=\frac{h'^2}{4}+a'b',
\qquad
\frac{\Acoef(t)}4+\Bcoef(t)
=\frac{1-\pi t\cot(\pi t)}{\pi^2t^2}
\]
give
\[
\pi^2T
=\int_0^1\bigl(1-\pi h\cot(\pi h)\bigr)
\left(\frac{h'}h\right)^2z^2
+\pi^2\int_0^1\Acoef(h)a'b'z^2.
\]
Both integrals are finite. Indeed, the bound
\eqref{eq:a'_and_b'} and Hardy's inequality \eqref{eq:Hardy} give
\[
\int_0^1\Acoef(h)|a'b'|z^2\, dx
\le C_r\int_0^1\frac{z(x)^2}{\min\{x,1-x\}^2}\,dx
\le C_r\int_0^1|z'|^2\, dx<\infty,
\]
and then the first integral is finite since $T<+\infty$. 

Write $\Gflux:=M-\pi^2(Z+T)$.
From \eqref{eq:positive-identity}, we obtain
\begin{align}\label{eq:master-exact}
\Gflux=E+V
+\int_0^1\left(\pi h\cot(\pi h)-\frac14\right)
\left(\frac{h'}h\right)^2z^2
-\pi^2\int_0^1\Acoef(h)a'b'z^2.
\end{align}
Since $\Acoef$ is positive and increasing on $(0,1)$ and $0\le h\le r$,
we have
\[
\int_0^1\Acoef(h)a'b'z^2
\le\Acoef(r)\int_0^1(a'b')_+z^2
=\Acoef(r)S.
\]
By Lemma~\ref{lem:universal-shear}, we obtain
\begin{align}\label{eq:rho-remainder}
\Gflux\ge(1-\kappa_r)(E+V)
+\int_0^1\left(\pi h\cot(\pi h)-\frac14\right)
\left(\frac{h'}h\right)^2z^2,
\end{align}
where
\[
\kappa_r=\frac{5}{14}\pi^2\Acoef(r)r^2
=\frac{5\pi r}{7}\,\tan\frac{\pi r}{2}.
\]
Since $f(s)=s\cot s$ is strictly decreasing on $(0,\pi/2)$ and $h\le r\le\rho$, we obtain
\[
\pi h\cot(\pi h)-\frac14\ge0.
\]
Moreover, $\kappa_r$ is increasing in $r$, and
Lemma~\ref{lem:rho-constants} gives $\kappa_r\le\kappa_\rho<1$.
It follows that 
\[
\Gflux\ge(1-\kappa_\rho)(E+V)\ge 0,
\]
that is, 
\begin{equation}\label{eq:MgeZ+T}
	M\ge\pi^2(Z+T).
\end{equation}
Since the
mean-zero function $u$ was arbitrary, Lemma~\ref{lem:transverse} yields
$\mu_1(\Omega)\ge\pi^2$.

It remains to show the equality case.
Suppose that $\mu_1(\Omega)=\pi^2$. A bounded convex domain with nonempty
interior is a Lipschitz domain, so the embedding
$H^1(\Omega)\hookrightarrow L^2(\Omega)$ is compact, and the infimum in
\eqref{eq:variational} is attained. Let $u$ be a nonzero first
eigenfunction, which has mean zero. Then
\[
\int_\Omega|\nabla u|^2=\pi^2(M+Q).
\]
By \eqref{eq:zero-sectional-mean}, we have $M>0$. Substituting the
preceding identity into \eqref{eq:flux-energy}, we obtain
\[
M^2\le\pi^2M(Z+T),
\]
that is 
\[
M\le\pi^2(Z+T).
\]
Combining with \eqref{eq:MgeZ+T}, we have $\Gflux=0$.
Thus, every nonnegative term on the right-hand side of
\eqref{eq:rho-remainder} vanishes. Since $1-\kappa_r>0$, this implies
\begin{equation*}
E=V_a=V_b=0,
\end{equation*}
and
\begin{equation}\label{eq:h_constant}
\int_0^1\left(\pi h\cot(\pi h)-\frac14\right)
\left(\frac{h'}h\right)^2z^2=0.
\end{equation}
The identity $E=0$ shows that $z/\phi$ is constant on $(0,1)$, so that
\[
z(x)=\lambda\sin(\pi x).
\]
The constant $\lambda$ is nonzero, since otherwise $F=\sqrt h\,z=0$ and
$F'=-hm$ would give $M=0$. Therefore $z^2/h$ is continuous and strictly
positive in $(0,1)$. Since $\nu_a$ and
$\nu_b$ are nonnegative and $V_a=V_b=0$, it follows that
\[
\nu_a=\nu_b=0\quad\text{on }(0,1).
\]
Thus $a$, $b$ and $h$ are affine.

Since $h(x)\le r\le\rho$, \eqref{eq:h_constant} gives
\[
\left(\pi h\cot(\pi h)-\frac14\right)
\left(\frac{h'}h\right)^2z^2=0
\qquad\text{on }(0,1),
\]
which implies $h$ is constant
and $a'=b'$ in $(0,1)$. The convexity of $a$ and $-b$ gives $a''=b''=0$ and $a'=b'=k$ for some constant $k$. 
By Lemma~\ref{lem:contacts}, $k=0$, so the boundary functions $a$ and $b$ are constant. 
Since every side of $R_r$ intersects $\overline\Omega$, 
it follows that $a\equiv0$ and $b\equiv r$, and hence $\Omega=\Int R_r$. Conversely, $\mu_1(\Int R_r)=\pi^2$ because $r<1$.
Scaling back proves the theorem.
\end{proof}

\input{appendix}

\bibliographystyle{amsalpha}
\bibliography{ref}
\end{document}

%% file: appendix.tex
\appendix
\section{Near-optimality of the aspect-ratio bound}\label{sec:near-optimality}

In this appendix, we show that the threshold $\rho$ in
Theorem~\ref{thm:rho} cannot be replaced by any number larger than
$r_*\approx0.457993$. This implies that the optimal threshold lies in $[\rho,r_*]$. The counterexamples are trapezoids obtained from the rectangle $R_r$ by slightly tilting its top side. Their construction
is governed by the following function. For $0<r<1/2$, define
\begin{equation}\label{eq:critical-coefficient}
	\mathcal C(r)=\frac{8r}{\pi}
	\left\{\cot(\pi r)
	-2\sum_{k=1}^{\infty}
	\frac{\coth\!\left(\pi r\sqrt{4k^2-1}\right)}
	{(4k^2-1)^{5/2}}\right\}.
\end{equation}

\begin{lemma}\label{lem:app-critical-coefficient}
	The function $\mathcal C$ is continuous and strictly decreasing on
	$(0,1/2)$. It has a unique zero $r_*\in(0,1/2)$, and
	$\mathcal C(r)<0$ for $r_*<r<1/2$.
\end{lemma}

\begin{proof}
	Since $\coth$ is decreasing on $(0,\infty)$, the series in
	\eqref{eq:critical-coefficient} converges uniformly on compact
	subintervals of $(0,1/2]$. Thus, $\mathcal C$ is a continuous function on
	$(0,1/2]$.
	Set $t_k=\pi r\sqrt{4k^2-1}$ and rewrite \eqref{eq:critical-coefficient} as
	\[
	\mathcal C(r)=\frac{8}{\pi^2}
	\left\{\pi r\cot(\pi r)
	-2\sum_{k=1}^{\infty}\frac{t_k\coth t_k}{(4k^2-1)^{3}}\right\}.
	\]
	Recall that $t\cot t$ is strictly decreasing on $(0,\pi/2)$. Moreover,
	since for $t>0$,
	\[
	\frac{d}{dt}(t\coth t)
	=\frac{\sinh t\cosh t-t}{\sinh^2t}>0,
	\]
	it follows that
	$\mathcal C$ is strictly decreasing on $(0,1/2]$.
	
	Since $t\cot t\to 1$ and $t\coth t\to 1$ as $t\to0$, we have
	\[
	\lim_{r\to0}\mathcal C(r)
	=\frac8{\pi^2}\left\{1-2\sum_{k=1}^{\infty}
	\frac1{(4k^2-1)^3}\right\}>0.
	\]
	On the other hand,
	$\cot(\pi/2)=0$, so $\mathcal C(1/2)<0$. Since $\mathcal C$ is
	continuous and strictly decreasing on $(0,1/2]$, it has a unique zero
	$r_*\in(0,1/2)$, and $\mathcal C(r)<0$ for $r_*<r<1/2$.
\end{proof}

\begin{proposition}\label{prop:app-trapezoids}
	Let $r_*<r<1/2$. For $0<s<1$, consider the convex trapezoid
	\[
	\Omega_{r,s}
	=\{(x,y):0<x<1,\ 0<y<r(1-sx)\}.
	\]
	Then for every $0<s<1$, the rectangle $R_r=[0,1]\times[0,r]$ is a
	minimum-area enclosing rectangle of $\Omega_{r,s}$, and
	\[
	d_H(\overline{\Omega_{r,s}},R_r)\le rs,
	\]
	where $d_H$ denotes the Hausdorff distance between compact sets.
	Moreover,
	\[
	\mu_1(\Omega_{r,s})<\pi^2
	\]
	for all sufficiently small $s>0$.
\end{proposition}

The test functions in the proof of
Proposition~\ref{prop:app-trapezoids} are perturbations of the first
nonconstant Neumann eigenfunction of the rectangle. Put
\[
\mathcal D_r=(0,1)\times(0,r),
\qquad
\Psi(x)=\sqrt2\cos(\pi x).
\]
Then $-\Delta\Psi=\pi^2\Psi$ and $\int_{\mathcal D_r}\Psi^2=r$. On
$\Omega_{r,s}$, however, $\Psi$ does not have mean zero, and its normal
derivative does not vanish on the tilted side. We eliminate both obstructions by subtracting $s\chi_r$, 
where $\chi_r$ is the function on $\mathcal D_r$ given by the following lemma.

\begin{lemma}\label{lem:app-correction}
	For each $0<r<1/2$, there is a function
	$\chi_r\in C^1(\overline{\mathcal D_r})$ such that
	\begin{equation}\label{eq:app-correction-weak}
		\int_{\mathcal D_r}
		\bigl(\nabla\chi_r\cdot\nabla\eta-\pi^2\chi_r\eta\bigr)
		=r\int_0^1\Psi'(x)\eta(x,r)\,dx
	\end{equation}
	for every $\eta\in H^1(\mathcal D_r)$. Moreover,
	\begin{equation}\label{eq:app-correction-identities}
		\begin{aligned}
			&\int_{\mathcal D_r}\chi_r
			=-\int_{\mathcal D_r}x\Psi
			=\frac{2\sqrt2\,r}{\pi^2},\\
			&\int_{\mathcal D_r}\bigl(|\nabla\chi_r|^2-\pi^2\chi_r^2\bigr)
			=-r\mathcal C(r).
		\end{aligned}
	\end{equation}
\end{lemma}

\begin{proof}
	On $(0,1)$, the function $\Psi'(x)=-\sqrt2\,\pi\sin(\pi x)$ has the
	cosine expansion
	\[
	\Psi'(x)=-2\sqrt2
	+4\sqrt2\sum_{k=1}^{\infty}
	\frac{\cos(2k\pi x)}{4k^2-1},
	\]
	which converges absolutely and uniformly on $[0,1]$. Set
	\[
	\omega_k=\pi\sqrt{4k^2-1},\qquad k\ge1,
	\]
	and define
	\begin{equation}\label{eq:app-correction-series}
		\chi_r(x,y)
		=\frac{2\sqrt2\,r\cos(\pi y)}{\pi\sin(\pi r)}
		+4\sqrt2\,r\sum_{k=1}^{\infty}
		\frac{\cos(2k\pi x)\cosh(\omega_k y)}
		{(4k^2-1)\omega_k\sinh(\omega_k r)}.
	\end{equation}
	Since the $k$-th term of the
	series in \eqref{eq:app-correction-series} and its first derivatives are
	bounded on $\overline{\mathcal D_r}$ by $C_rk^{-2}$, we have
	$\chi_r\in C^1(\overline{\mathcal D_r})$. Note that every term in
	\eqref{eq:app-correction-series} satisfies $-\Delta v=\pi^2v$. Thus,
	\[
	-\Delta\chi_r=\pi^2\chi_r\quad\text{in }\mathcal D_r.
	\]
	Differentiating termwise, we see that $\partial_\nu\chi_r=0$ on the
	bottom and vertical sides of $\mathcal D_r$. On the top side, the
	cosine expansion of $\Psi'$ gives
	\[
	\partial_\nu\chi_r(x,r)=\partial_y\chi_r(x,r)=r\Psi'(x).
	\]
	Integrating by parts, we obtain \eqref{eq:app-correction-weak}.
	
	Taking $\eta=1$ in \eqref{eq:app-correction-weak}, we get
	\[
	-\pi^2\int_{\mathcal D_r}\chi_r
	=r\bigl(\Psi(1)-\Psi(0)\bigr)
	=-2\sqrt2\,r.
	\]
	On the other hand, a direct computation gives
	\[
	\int_{\mathcal D_r}x\Psi
	=\sqrt2\,r\int_0^1x\cos(\pi x)\,dx
	=-\frac{2\sqrt2\,r}{\pi^2}.
	\]
	This proves the first identity in \eqref{eq:app-correction-identities}.
	Taking $\eta=\chi_r$ in \eqref{eq:app-correction-weak}, we get
	\[
	\int_{\mathcal D_r}\bigl(|\nabla\chi_r|^2-\pi^2\chi_r^2\bigr)
	=r\int_0^1\Psi'(x)\chi_r(x,r)\,dx.
	\]
	By the orthogonality of the functions $1$ and $\cos(2k\pi x)$
	in $L^2(0,1)$, we have
	\[
	\begin{aligned}
		\int_0^1\Psi'(x)\chi_r(x,r)\,dx
		&=-\frac{8r}{\pi}\cot(\pi r)
		+\frac{16r}{\pi}\sum_{k=1}^{\infty}
		\frac{\coth(\omega_k r)}{(4k^2-1)^{5/2}}\\
		&=-\mathcal C(r),
	\end{aligned}
	\]
	which yields the second identity.
\end{proof}

\begin{proof}[Proof of Proposition~\ref{prop:app-trapezoids}]
	We first show that $R_r$ is a minimum-area enclosing rectangle of
	$\Omega_{r,s}$. Since $r(1-sx)\ge r(1-x)$ for $0\le x\le1$, the closure
	of $\Omega_{r,s}$ contains the triangle with vertices
	\[
	(0,0),\qquad (1,0),\qquad (0,r),
	\]
	whose area is $r/2$. Note that a triangle contained in a rectangle has at most
	half the area of the rectangle. Hence, every rectangle containing
	$\overline{\Omega_{r,s}}$ has area at least $r$, which implies $R_r$ is a minimum-area enclosing rectangle. 
	Moreover, each point of
	$R_r$ lies within vertical distance $rsx\le rs$ of
	$\overline{\Omega_{r,s}}$, which gives the Hausdorff estimate.
	
	We now turn to the eigenvalue. Since $r>r_*$,
	Lemma~\ref{lem:app-critical-coefficient} gives $\mathcal C(r)<0$. Let
	$\chi_r$ be the function from Lemma~\ref{lem:app-correction}, and
	consider
	\[
	U_s(x,y)=\Psi(x)-s\chi_r(x,y).
	\]
	Here and below, $O_r(s^k)$ denotes a quantity whose absolute value is
	bounded by $C_rs^k$ for $0<s<1$.
	
	Since $-\Delta\Psi=\pi^2\Psi$ and $\Psi'(0)=\Psi'(1)=0$, it follows that for every $\eta\in H^1(\Omega_{r,s})$,
	\begin{equation}\label{eq:app-Psi-weak}
		\int_{\Omega_{r,s}}
		\bigl(\nabla\Psi\cdot\nabla\eta-\pi^2\Psi\eta\bigr)
		=rs\int_0^1\Psi'(x)\eta(x,r(1-sx))\,dx.
	\end{equation}
	For $\eta=\Psi$, \eqref{eq:app-Psi-weak} gives
	\[
	\int_{\Omega_{r,s}}
	\bigl(|\nabla\Psi|^2-\pi^2\Psi^2\bigr)
	=0.
	\]
	Expanding the quadratic form and applying \eqref{eq:app-Psi-weak} with
	$\eta=\chi_r$, we obtain
	\[
	\int_{\Omega_{r,s}}\bigl(|\nabla U_s|^2-\pi^2U_s^2\bigr)
	=s^2\int_{\Omega_{r,s}}
	\bigl(|\nabla\chi_r|^2-\pi^2\chi_r^2\bigr)
	-2rs^2\int_0^1\Psi'(x)\chi_r(x,r(1-sx))\,dx.
	\]
	Since $\chi_r$ and its first derivatives are bounded on
	$\overline{\mathcal D_r}$ and $|\mathcal D_r\setminus\Omega_{r,s}|=rs/2$,
	replacing $\Omega_{r,s}$ by $\mathcal D_r$ in the first integral and
	$r(1-sx)$ by $r$ in the second changes the right-hand side only by
	$O_r(s^3)$. Moreover, \eqref{eq:app-correction-weak} with
	$\eta=\chi_r$ gives
	\[
	r\int_0^1\Psi'(x)\chi_r(x,r)\,dx
	=\int_{\mathcal D_r}\bigl(|\nabla\chi_r|^2-\pi^2\chi_r^2\bigr).
	\]
	Therefore, by \eqref{eq:app-correction-identities},
	\begin{equation}\label{eq:app-energy}
		\begin{aligned}
			\int_{\Omega_{r,s}}\bigl(|\nabla U_s|^2-\pi^2U_s^2\bigr)
			&=-s^2\int_{\mathcal D_r}
			\bigl(|\nabla\chi_r|^2-\pi^2\chi_r^2\bigr)+O_r(s^3)\\
			&=r\mathcal C(r)s^2+O_r(s^3).
		\end{aligned}
	\end{equation}
	
	Since $\int_0^1\Psi=0$, the first identity in
	\eqref{eq:app-correction-identities} gives
	\[
	\int_{\Omega_{r,s}}\Psi
	=-s\int_{\mathcal D_r}x\Psi
	=s\int_{\mathcal D_r}\chi_r.
	\]
	Consequently,
	\[
	\int_{\Omega_{r,s}}U_s
	=s\int_{\mathcal D_r\setminus\Omega_{r,s}}\chi_r
	=O_r(s^2).
	\]
	Let
	\[
	\widehat U_s
	=U_s-\frac1{|\Omega_{r,s}|}\int_{\Omega_{r,s}}U_s
	\]
	be the mean-zero part of $U_s$. Then $\nabla\widehat U_s=\nabla U_s$
	and
	\[
	\int_{\Omega_{r,s}}\widehat U_s^2
	=\int_{\Omega_{r,s}}U_s^2
	-\frac1{|\Omega_{r,s}|}\left(\int_{\Omega_{r,s}}U_s\right)^2
	=\int_{\Omega_{r,s}}U_s^2
	+O_r(s^4).
	\]
	Hence, \eqref{eq:app-energy} gives
	\[
	\int_{\Omega_{r,s}}|\nabla\widehat U_s|^2
	-\pi^2\int_{\Omega_{r,s}}\widehat U_s^2
	=r\mathcal C(r)s^2+O_r(s^3),
	\]
	which is negative for all sufficiently small $s>0$, since
	$\mathcal C(r)<0$. On the other hand, as $s\to0$,
	\[
	\int_{\Omega_{r,s}}\widehat U_s^2
	\longrightarrow\int_{\mathcal D_r}\Psi^2=r>0.
	\]
	Therefore, for all sufficiently small $s>0$, the function $\widehat U_s$ is
	nonzero, has mean zero on $\Omega_{r,s}$, and has Rayleigh quotient
	less than $\pi^2$. The variational characterization
	\eqref{eq:variational} then gives $\mu_1(\Omega_{r,s})<\pi^2$.
\end{proof}

%% file: ref.bib
@article {FK25,
    AUTHOR = {Freitas, Pedro and Kennedy, James B.},
     TITLE = {On domain monotonicity of {N}eumann eigenvalues of convex
              domains},
   JOURNAL = {Proc. Amer. Math. Soc.},
  FJOURNAL = {Proceedings of the American Mathematical Society},
    VOLUME = {153},
      YEAR = {2025},
    NUMBER = {12},
     PAGES = {5315--5328},
      ISSN = {0002-9939,1088-6826},
   MRCLASS = {35P15 (35J05 35J25)},
  MRNUMBER = {4989645},
MRREVIEWER = {Paolo\ Luzzini},
       DOI = {10.1090/proc/17394},
}

@article {PW60,
    AUTHOR = {Payne, L. E. and Weinberger, H. F.},
     TITLE = {An optimal {P}oincar\'e{} inequality for convex domains},
   JOURNAL = {Arch. Rational Mech. Anal.},
  FJOURNAL = {Archive for Rational Mechanics and Analysis},
    VOLUME = {5},
      YEAR = {1960},
     PAGES = {286--292},
      ISSN = {0003-9527},
   MRCLASS = {35.00},
  MRNUMBER = {117419},
MRREVIEWER = {I.\ Stakgold},
       DOI = {10.1007/BF00252910},
}

@misc {WX24,
  title={The first {N}eumann eigenvalue and the width},
  author={Wang, Haibin and Xu, Guoyi},
  note={arXiv preprint arXiv:2407.13984},
  year={2024}
}

@article {Szego54,
    AUTHOR = {Szeg\"o, G.},
     TITLE = {Inequalities for certain eigenvalues of a membrane of given area},
   JOURNAL = {J. Rational Mech. Anal.},
  FJOURNAL = {Journal of Rational Mechanics and Analysis},
    VOLUME = {3},
      YEAR = {1954},
     PAGES = {343--356},
      ISSN = {1943-5282,1943-5290},
   MRCLASS = {36.0X},
  MRNUMBER = {61749},
MRREVIEWER = {E.\ T.\ Copson},
       DOI = {10.1512/iumj.1954.3.53017},
}

@article {Weinberger56,
    AUTHOR = {Weinberger, H. F.},
     TITLE = {An isoperimetric inequality for the {$N$}-dimensional free
              membrane problem},
   JOURNAL = {J. Rational Mech. Anal.},
  FJOURNAL = {Journal of Rational Mechanics and Analysis},
    VOLUME = {5},
      YEAR = {1956},
     PAGES = {633--636},
      ISSN = {1943-5282,1943-5290},
   MRCLASS = {52.0X},
  MRNUMBER = {79286},
       DOI = {10.1512/iumj.1956.5.55021},
}

@misc{ABF24,
  title={The geometric size of the fundamental gap},
  author={Amato, Vincenzo and Bucur, Dorin and Fragal{\`a}, Ilaria},
  note={arXiv preprint arXiv:2407.01341},
  year={2024}
}

@misc {ABF26,
    AUTHOR = {Amato, Vincenzo and Bucur, Dorin and Fragal\`a, Ilaria},
     TITLE = {A sharp quantitative nonlinear {P}oincar\'e inequality on
              convex domains},
   JOURNAL = {Ann. Inst. H. Poincar\'e C Anal. Non Lin\'eaire},
HOWPUBLISHED = {Ann. Inst. H. Poincar\'e C Anal. Non Lin\'eaire},
  FJOURNAL = {Annales de l'Institut Henri Poincar\'e C. Analyse Non Lin\'eaire},
      YEAR = {2026},
      NOTE = {Published online first 24 June 2026,
              doi:10.4171/AIHPC/185},
       DOI = {10.4171/AIHPC/185},
}

@article {CFHLLS25,
    AUTHOR = {Cavallina, Lorenzo and Funano, Kei and Henrot, Antoine and
              Lemenant, Antoine and Lucardesi, Ilaria and Sakaguchi, Shigeru},
     TITLE = {Two extremum problems for {N}eumann eigenvalues},
   JOURNAL = {J. Anal. Math.},
  FJOURNAL = {Journal d'Analyse Math\'ematique},
    VOLUME = {155},
      YEAR = {2025},
    NUMBER = {2},
     PAGES = {657--697},
      ISSN = {0021-7670,1565-8538},
   MRCLASS = {49R05 (31B10 35J05 35P15 49Q10)},
  MRNUMBER = {4905324},
MRREVIEWER = {Marcos\ Montenegro},
       DOI = {10.1007/s11854-025-0366-3},
}

@article {Kroger99,
    AUTHOR = {Kr{\"o}ger, Pawel},
     TITLE = {On upper bounds for high order {N}eumann eigenvalues of convex
              domains in {E}uclidean space},
   JOURNAL = {Proc. Amer. Math. Soc.},
  FJOURNAL = {Proceedings of the American Mathematical Society},
    VOLUME = {127},
      YEAR = {1999},
    NUMBER = {6},
     PAGES = {1665--1669},
      ISSN = {0002-9939,1088-6826},
   MRCLASS = {35P15 (34L15 35J25)},
  MRNUMBER = {1486739},
MRREVIEWER = {Julian\ Edward},
       DOI = {10.1090/S0002-9939-99-04804-2},
}

@article {Ni2013,
    AUTHOR = {Ni, Lei},
     TITLE = {Estimates on the modulus of expansion for vector fields
              solving nonlinear equations},
   JOURNAL = {J. Math. Pures Appl. (9)},
  FJOURNAL = {Journal de Math\'{e}matiques Pures et Appliqu\'{e}es.
              Neuvi\`eme S\'{e}rie},
    VOLUME = {99},
      YEAR = {2013},
    NUMBER = {1},
     PAGES = {1--16},
      ISSN = {0021-7824,1776-3371},
   MRCLASS = {53C21 (35B50 35P15)},
  MRNUMBER = {3003280},
MRREVIEWER = {Qiaoling\ Xia},
       DOI = {10.1016/j.matpur.2012.05.009},
       URL = {https://doi.org/10.1016/j.matpur.2012.05.009},
}

@inproceedings{LY80,
  author = {Li, Peter and Yau, Shing Tung},
  title = {Estimates of eigenvalues of a compact {R}iemannian manifold},
  booktitle = {Geometry of the {L}aplace operator ({P}roc. {S}ympos. {P}ure {M}ath., {U}niv. {H}awaii, {H}onolulu, {H}awaii, 1979)},
  series = {Proc. Sympos. Pure Math., XXXVI},
  pages = {205--239},
  publisher = {Amer. Math. Soc., Providence, R.I.},
  year = {1980},
  mrclass = {58G25 (53C20)},
  mrnumber = {573435},
  mrreviewer = {P. G\"{u}nther},
}

@article{ZY84,
  author = {Zhong, Jia Qing and Yang, Hong Cang},
  title = {On the estimate of the first eigenvalue of a compact {R}iemannian manifold},
  journal = {Sci. Sinica Ser. A},
  fjournal = {Scientia Sinica. Series A. Mathematical, Physical, Astronomical \& Technical Sciences},
  volume = {27},
  year = {1984},
  number = {12},
  pages = {1265--1273},
  issn = {0253-5831},
  mrclass = {58G25},
  mrnumber = {794292},
  mrreviewer = {Domenico Perrone},
}

@article{AC13,
  author = {Andrews, Ben and Clutterbuck, Julie},
  title = {Sharp modulus of continuity for parabolic equations on manifolds and lower bounds for the first eigenvalue},
  journal = {Anal. PDE},
  fjournal = {Analysis \& PDE},
  volume = {6},
  year = {2013},
  number = {5},
  pages = {1013--1024},
  issn = {2157-5045},
  mrclass = {35R01 (35B65 35K59 35P15)},
  mrnumber = {3125548},
  mrreviewer = {Rodica Luca},
  doi = {10.2140/apde.2013.6.1013},
  url = {http://dx.doi.org/10.2140/apde.2013.6.1013},
}

@article{BQ00,
  author = {Bakry, Dominique and Qian, Zhongmin},
  title = {Some new results on eigenvectors via dimension, diameter, and {R}icci curvature},
  journal = {Adv. Math.},
  fjournal = {Advances in Mathematics},
  volume = {155},
  year = {2000},
  number = {1},
  pages = {98--153},
  issn = {0001-8708},
  mrclass = {58J50 (53C21)},
  mrnumber = {1789850},
  mrreviewer = {Olivier Druet},
  doi = {10.1006/aima.2000.1932},
  url = {https://doi.org/10.1006/aima.2000.1932},
}

@article {LW21JGA,
    AUTHOR = {Li, Xiaolong and Wang, Kui},
     TITLE = {Sharp lower bound for the first eigenvalue of the weighted
              {$p$}-{L}aplacian {I}},
   JOURNAL = {J. Geom. Anal.},
  FJOURNAL = {Journal of Geometric Analysis},
    VOLUME = {31},
      YEAR = {2021},
    NUMBER = {8},
     PAGES = {8686--8708},
      ISSN = {1050-6926,1559-002X},
   MRCLASS = {58J50 (35P15 35P30)},
  MRNUMBER = {4293945},
MRREVIEWER = {He-Jun\ Sun},
       DOI = {10.1007/s12220-021-00613-4},
       URL = {https://doi.org/10.1007/s12220-021-00613-4},
}

@article {LW21MRL,
    AUTHOR = {Li, Xiaolong and Wang, Kui},
     TITLE = {Sharp lower bound for the first eigenvalue of the weighted
              {$p$}-{L}aplacian {II}},
   JOURNAL = {Math. Res. Lett.},
  FJOURNAL = {Mathematical Research Letters},
    VOLUME = {28},
      YEAR = {2021},
    NUMBER = {5},
     PAGES = {1459--1479},
      ISSN = {1073-2780,1945-001X},
   MRCLASS = {58J50 (35P15)},
  MRNUMBER = {4471716},
MRREVIEWER = {He-Jun\ Sun},
       DOI = {10.4310/mrl.2021.v28.n5.a8},
       URL = {https://doi.org/10.4310/mrl.2021.v28.n5.a8},
}

@article {LW21TAMS,
    AUTHOR = {Li, Xiaolong and Wang, Kui},
     TITLE = {Lower bounds for the first eigenvalue of the {L}aplacian on
              {K}\"{a}hler manifolds},
   JOURNAL = {Trans. Amer. Math. Soc.},
  FJOURNAL = {Transactions of the American Mathematical Society},
    VOLUME = {374},
      YEAR = {2021},
    NUMBER = {11},
     PAGES = {8081--8099},
      ISSN = {0002-9947},
   MRCLASS = {35P15 (53C55)},
  MRNUMBER = {4328692},
       DOI = {10.1090/tran/8434},
       URL = {https://doi.org/10.1090/tran/8434},
}

@article{AN12,
  author = {Andrews, Ben and Ni, Lei},
  title = {Eigenvalue comparison on {B}akry-{E}mery manifolds},
  journal = {Comm. Partial Differential Equations},
  fjournal = {Communications in Partial Differential Equations},
  volume = {37},
  year = {2012},
  number = {11},
  pages = {2081--2092},
  issn = {0360-5302},
  coden = {CPDIDZ},
  mrclass = {35K05 (35B51 35J10 35P15)},
  mrnumber = {3005536},
  doi = {10.1080/03605302.2012.668602},
  url = {http://dx.doi.org/10.1080/03605302.2012.668602},
}

@article {ZW2017,
    AUTHOR = {Zhang, Yuntao and Wang, Kui},
     TITLE = {An alternative proof of lower bounds for the first eigenvalue
              on manifolds},
   JOURNAL = {Math. Nachr.},
  FJOURNAL = {Mathematische Nachrichten},
    VOLUME = {290},
      YEAR = {2017},
    NUMBER = {16},
     PAGES = {2708--2713},
      ISSN = {0025-584X,1522-2616},
   MRCLASS = {58C40 (35B50 35J05 35J20 35P15)},
  MRNUMBER = {3722506},
MRREVIEWER = {Henrik\ Matthiesen},
       DOI = {10.1002/mana.201600388},
       URL = {https://doi.org/10.1002/mana.201600388},
}
